\documentclass[11pt,reqno,a4paper]{amsart}
\usepackage{latexsym,amssymb}
\usepackage{natbib}
\usepackage[ansinew]{inputenc}
\usepackage[T1]{fontenc}
\usepackage{times,mathptmx}
\usepackage{color}
\usepackage{hyperref}
\theoremstyle{plain}

\newtheorem{PropIntro}{Proposition}
\newtheorem{ThmIntro}[PropIntro]{Theorem}

\newtheorem{RemIntro}[PropIntro]{Remark}

\newtheorem{Prop}{Proposition}[section]

\newtheorem{Lem}[Prop]{Lemma}
\newtheorem{rmq}[Prop]{Remark}

\newtheorem{Rem}[Prop]{Remark}

\newcommand{\IR}{\mathbb{R}}

\newcommand{\C}{\mathcal{C}}

\newcommand{\N}{{{\mathcal N}}}

\newcommand{\E}{{{\mathcal E}}}

\newcommand{\abs}[1]{\mathopen\vert#1\mathclose\vert}

\newcommand{\norm}[1]{\mathopen\Vert#1\mathclose\Vert}
\newcommand{\intd}{\,{\mathrm d}}
\renewcommand{\phi}{\varphi}
\renewcommand{\epsilon}{\varepsilon}

\renewcommand{\le}{\leqslant}
\renewcommand{\ge}{\geqslant}
\renewcommand{\subset}{\subseteq}

\definecolor{umhblue}{rgb}{.69,.75,.86}
\definecolor{umhdarkblue}{rgb}{.15,.15,.53}
\definecolor{umhred}{rgb}{0.4,.0,.0}
\definecolor{mygreen}{RGB}{0, 165, 0}
\definecolor{definition}{rgb}{0.45, 0.61, 0.96}
\definecolor{problem}{rgb}{0.84,0.5,0}

\begin{document}

\title[Lane Emden problems and Liouville equations]{Lane Emden
problems with large exponents and singular Liouville equations}
\author[M.~Grossi, C.~Grumiau, F.~Pacella]{Massimo Grossi, Christopher
  Grumiau, Filomena Pacella}

\address{
Dipartimento di Matematica\\
   Universita' di Roma  ``Sapienza''\\
  P.~le A.~Moro 2, 00185 Roma, Italy}
\email{grossi@mat.uniroma.it (M.~Grossi)}
\email{pacella@mat.uniroma.it (F.~Pacella)}

\address{
  Institut de Math{\'e}matique\\
  Universit{\'e} de Mons\\
  20, Place du Parc, B-7000 Mons, Belgium}
\email{Christopher.Grumiau@umons.ac.be (C.~Grumiau)}

\begin{abstract}
We consider the  Lane-Emden Dirichlet problem
\begin{equation*}
\left\{
\begin{aligned}
-\Delta u&= \abs{ u}^{p-1}u,&&\text{ in } B, \\
u&=0,&&\text{ on } \partial B,
\end{aligned}
\right.
\end{equation*}
where  $p>1$ and $B$ denotes the unit ball in $\IR^2$. 
We study the asymptotic behavior of the least energy nodal radial
solution $u_p$, as  $p\rightarrow +\infty$. Assuming w.l.o.g.\
that $u_p(0) < 0$,  we prove that a suitable rescaling of the
negative part $u_p^-$ converges to the unique regular solution of
the Liouville equation in $\IR^2$, while a suitable rescaling of
the positive part $u_p^+$ converges to a (singular) solution of a
singular Liouville equation in $\IR^2$.  We also get exact
asymptotic values for the $L^\infty$-norms of $u_p^-$ and $u_p^+$,
as well as an asymptotic estimate of the energy. Finally, we have
that the nodal line $\N_p:=\{ x\in B : \abs{x}= r_p\}$ shrinks to
a point and we compute the rate of convergence of $r_p$.
\end{abstract}



\thanks{Partially supported by PRIN-2009-WRJ3W7 grant.\\
This work has been done while the second author was visiting
  the Mathematics Department of the University of Roma ``
  Sapienza'' supported by INDAM-GNAMPA.  He also acknowledges  the program
``Qualitative study of
 solutions of variational elliptic partial differerential equations. Symmetries,
bifurcations, singularities, multiplicity and numerics'' of the
FNRS, project 2.4.550.10.F of the Fonds de la Recherche
Fondamentale Collective for the partial support.}

\maketitle

\section{Introduction} \label{Section-Intro}

We consider the superlinear elliptic Dirichlet problem
\begin{equation*}
\tag{\protect{$\mathcal{P}_p$}} \label{pblP} \left\{
\begin{aligned}
-\Delta u&= \abs{ u}^{p-1}u,&&\text{ in } B, \\
u&=0,&&\text{ on } \partial B,
\end{aligned}
\right.
\end{equation*}
where $B$ is the unit ball  in $\IR^{2}$ and
$p>1$.\\

In this paper,  we are interested in studying the asymptotic
behavior, as $p\to +\infty$, of the least energy sign changing
radial solution of~\eqref{pblP} which will be denoted by $u_p$.
This solution has two nodal regions and it has been proved
in~\cite{aftalion} that it is not the least energy nodal solution
of Problem~\eqref{pblP} in the whole space $H^1_0(B)$. Indeed,
$u_p$ has Morse index at least three while the least energy nodal
solution has Morse index two (see \cite{bartweth}) and its nodal
line touches the boundary (~\cite{aftalion}).

In our previous paper~\cite{GGP}, we have analyzed the asymptotic
behavior, as $p\to +\infty$, of low energy nodal solutions $w_p$
of Problem~\eqref{pblP}, i.e.\ solutions satisfying :

\begin{equation}
\label{lowenerg} p \int_{B}\abs{\nabla w_p}^2\to 16\pi e
\end{equation}
 as
$p\to +\infty$ in general bounded regular domains $\Omega$. Under
the additional condition~\eqref{B} that we recall later, we have
proved, among other results, that suitable rescaling of both
$w_p^+$ and $w_p^-$ converge, as $p\to +\infty$, to the regular
solution $U$ of the Liouville equation in $\IR^2$ with
$\int_{\IR^2}e^U <+\infty$. Moreover the $L^\infty$-norms
$\norm{w_p^\pm}_\infty$ converge to the same value $\sqrt e$ and,
in $\Omega$ for large $p$, $pw_p$ looks like the difference of two
Green functions, centered at the maximum and the minimum point of
$w_p$ which are far away from each other. So for this kind of
solutions the positive and negative part separate but have the
same profile and approach, after suitable rescaling, the same
solution of the same limit problem in $\IR^2$, as $p\to +\infty$.

A similar analysis was carried out in~\cite{benayed1} for low
energy nodal solutions of an almost critical problem in a bounded
domain $\Omega$ in $\IR^N$, $N\geq 3$, namely :

\begin{equation*}
\tag{\protect{$\mathcal{P}2_\varepsilon$}} \label{pblP2} \left\{
\begin{aligned}
-\Delta u&= \abs{ u}^{(2^*-2)-\varepsilon}u,&&\text{ in } \Omega, \\
u&=0,&&\text{ on } \partial \Omega,
\end{aligned}
\right.
\end{equation*}
where $2^* = \frac{2N}{N-2}$, $\varepsilon >0$ and $\varepsilon
\to 0$.

A complete classification of these solutions given
in~\cite{benayed1} together with an existence result
of~\cite{pistoiaweth} show the presence of both nodal solutions
concentrating  at two different points or at a single point. In
both cases, positive and negative points of the solution converge,
after rescaling, to the positive solution of the analogous problem
in $\IR^N$.

By the results of these two papers it is not difficult to deduce
that for the least energy nodal radial solution of ~\eqref{pblP2}
in the ball the positive and negative part concentrate at the
center of the ball as $\varepsilon\to 0$, approaching the
analogous problem in $\IR^N$ and carrying the same energy.

In view of these results, in studying the behavior of the least
energy nodal radial solution of~\eqref{pblP} as $p\to +\infty$,
one could expect a similar asymptotic behavior.

However that is not the case and we are able to show an
interesting new phenomenum : the positive and negative part of
$u_p$ concentrate at the center of the ball but the limit problems
for $u_p^+$ and $u_p^-$ are different. Indeed, assuming w.l.o.g.\
that $u_p(0)<0$, we prove that a suitable rescaling of $u_p^-$
converges to the regular solution of the Liouville equation in
$\IR^2$ while a suitable rescaling of $u_p^+$ converges to a
singular solution of a singular Liouville equation in $\IR^2$.
Moreover the limits of the $L^\infty$-norms of $u_p ^+$ and
$u_p^-$ are different as well as the
energies.\\
This shows that in our case the situation is more subtle and we think that it is
peculiar of the fact that we work in dimension $2$ (see also Remark \ref{rf})\\
To be more precise let us consider the following problems :

\begin{equation*}
\tag{\protect{$L1$}} \label{eqq1}
\left\{
\begin{aligned}
-\Delta u&= e^u, \text{ in } \IR^2.\\
\int_{\IR^2} e^u &<+\infty, \ \ \   u(0)=u'(0)=0,
\end{aligned}
\right.
\end{equation*}
which  has the  unique regular solution
\begin{equation}
\label{U} U(x) :=
\log\left(\frac{1}{(1+\frac{1}{8}\abs{x}^2)^2}\right)
\end{equation}
and, for $\delta_0$ being the Dirac measure at the origin,
\begin{equation*}
\tag{\protect{$L2$}} \label{eqq2} \left\{
\begin{aligned}
-\Delta u&= e^u + H \delta_0,&&\text{ in } \IR^2, \\
\int_{\IR^2}e^u &<+\infty,  \ \ && 
\end{aligned}
\right.
\end{equation*}
where $H$ is a constant, and whose radial solutions can be all
computed explicitely. Note that \eqref{eqq1} is the classical
Liouville equation in $\IR^2$ while~\eqref{eqq2}, after an easy
transformation, reduces to a limiting equation which appears in
the blow-up analysis of periodic vortices for the Chern-Simons
theory.

Then, denoting by
\begin{itemize}
\item $\N_p = \{x\in B : u_p(x)=0\} = \{ x\in B : \abs{x}= r_p\}$ the
  nodal circle;
\item $s_p\in (0,1)$ the value of the radius such that $u_p^+(x) =
  \norm{u_p^+}_\infty$ for $\abs{x}= s_p$;
\item $\left(\varepsilon_p^\pm\right)^{-2} = p \norm{u_p^\pm}_{\infty}^{p-1}$;
\end{itemize}
we have the following results :

\begin{ThmIntro}
\label{intro3} $i)$ Let  $z_p^-: B\left(0,
\frac{r_p}{\varepsilon_p^-}\right)\to \IR$ be defined as
\begin{equation}
\label{zp-} z_p^-(x)=
-\frac{p}{\norm{u_p^-}_{\infty}}(u_p^-(\varepsilon_p^- x)+u_p(0))
\end{equation}
then  $z_p^- \to -U$ in $\mathcal{C}^1_\text{loc}(\IR^2)$ as $p\to
+\infty$ ($U$ as in~\eqref{U}).

$ii)$ Let   $l = \lim_{p\to
  +\infty}\frac{s_p}{\varepsilon_p^+}$. Then, $l>0$ and defining the
one-variable function
\begin{equation}
\label{zp+} z_p^+(\abs{x})=z_p^+(r) = \frac{p}{u_p(s_p)}(u_p^+(s_p
+ \varepsilon_p^+ r ) - u_p(s_p))
\end{equation}
in the interval $\left(\frac{r_p-s_p}{\varepsilon_p^+} ,
  \frac{1-s_p}{\varepsilon_p^+}\right)$, we get
  $$z_p^+(\abs{x}-l)=
z_p^+(r-l) \to  Z_l(\abs{x}) = \log \left( \frac{2 \alpha^2
    \beta^\alpha \abs{x}^{\alpha-2}}{(\beta^\alpha +
    \abs{x}^\alpha)^2}\right)  $$  in $\mathcal{C}^1_\text{loc}(\IR^2
\setminus\{0\})$ as $p\to +\infty$ for $\alpha = \sqrt{2l^2
  +4}$ and $\beta = \left( \frac{\alpha +
    2}{\alpha-2}\right)^{1/\alpha}l$. Moreover,  $Z_l(\abs{x})$ is a radial
(distribution) solution of~\eqref{eqq2} for $H:=-
\int_{0}^le^{Z_l(s)} s \intd s$.
\end{ThmIntro}

\begin{ThmIntro}
\label{intro3bis} We have, as $p\to +\infty$,
\begin{equation}
\label{normup-} \norm{u_p^-}_\infty \to \frac{\sqrt e}{\Bar t}
e^{\frac{\Bar t}{2
    (\Bar t + \sqrt e)}}\approx 2.46,
\end{equation}
\begin{equation}
\label{normup+} \norm{u_p^+}_\infty \to e^ {\frac{\Bar t}{2(\Bar t
+ \sqrt e)}}\approx 1.17
\end{equation}
where $\Bar t\approx 0.7875$ is the unique root of the equation $
2\sqrt e \log t + t=0$, and
\begin{equation}
\label{ener} p \int_{B} \abs{\nabla u_p}^2  \to 8\pi e^{\frac{\Bar
t}{\Bar t +
  \sqrt e}}\left( \frac{e}{\Bar{t}^2} + 1 + \frac{2\sqrt e}{\Bar t}
\right)\approx 332.
\end{equation}
Moreover $pu_p(x)$ converges to $2\pi\gamma G(x,0)=\gamma\log|x|$ in
$C^1_{loc}\left(B_1\setminus\{0\}\right)$ where
$\gamma=\left(4+\frac{12\sqrt e}{\Bar t} \right)e^ {\frac{\Bar
t}{2(\Bar t + \sqrt e)}}$ and $G(x,0)$ is the Green function of the unit ball computed at the origin.

\end{ThmIntro}

To prove the previous results we start showing that the nodal line
shrinks to the origin. Later, as a consequence of the rescaling
argument and of the estimates needed to prove Theorem~\ref{intro3}
and Theorem~\ref{intro3bis} we also get the rate of convergence of
the ``nodal radius'' $r_p$. More precisely we show :

\begin{ThmIntro}
\label{intro4} We have
$$ r_p^{2/(p-1)} \to \frac{\Bar t}{e^{\frac{\Bar t}{2(\Bar t + \sqrt
      e)}}}\approx 0.67$$ for $\Bar t$ as in Theorem~\ref{intro3bis}.
\end{ThmIntro}

\begin{RemIntro}
As mentioned before we already know by~\cite{aftalion} that the
solutions $u_p$ are not the least energy nodal solutions
of~\eqref{pblP} in the whole $H^1_0(B)$. The convergence
result~\eqref{ener} shows that $u_p$ are not even ``low energy''
solutions in the sense that they do not satisfy~\eqref{lowenerg}.
Moreover, by comparing~\eqref{lowenerg} and \eqref{ener} we get
the exact difference between the limit energies of this kind of
solutions.
\end{RemIntro}
The proofs of the above theorems are quite long and technically
complicated. They follow from several delicate asymptotic
estimates on $u_p^-$ and $u_p^+$. In particular, the ones which
concern the asymptotic behavior of $u_p^+$, i.e.\ the part of
$u_p$ supported in an annulus are nontrivial and crucial for the
final results.

In the final part of the paper we shortly complete the analysis of
low-energy nodal solutions done in~\cite{GGP}, in the case of the
ball by considering solutions $w_p$ which satisfy \eqref{lowenerg}
and
\begin{equation}
\tag{\protect{$B$}} \label{B} \exists K>0 \text{ such that }
p\left|w_p(x_p^+) +
  w_p(x_p^-)\right|\le K
\end{equation}
with  $x_p^+$ and $x_p^-$ such that  $w_p(x_p^ \pm)= \pm
\norm{w_p^\pm}_\infty$.

We prove :
\begin{ThmIntro}
\label{intro1} If $w_p$ are sign changing solutions of
\eqref{pblP} with Morse index two, satisfying \eqref{lowenerg} and
\eqref{B} with $x^\pm = \lim_{p\to +\infty} x_p^\pm$ then $p w_p$
converge, as $p\to +\infty$ in $\C^1(\IR^2\setminus \{x^+,x^-\})$
to a function which is even  with respect to the diameter passing
through $x^+$ and $x^-$ and odd with respect to the orthogonal
diameter.
\end{ThmIntro}

The previous result is a further step in the direction of proving
that low energy solutions are indeed antisymmetric functions as it
is conjectured to be the case for least energy nodal solution in
the ball.

The paper is organized as follows. In Section $2$ we prove some
preliminary estimates on $u_p$ and $u_p^-$ and we show that the
nodal line shrinks to the origin. In  Section $3$ we prove the
important estimates on $u_p^+$. In Section $4$ we prove some
crucial lemmas about the limit values of $\norm{u_p^+}_{\infty}$
and $r_p$, as $p\to + \infty$. From these we deduce the proofs of
Theorem~\ref{intro3}, Theorem~\ref{intro3bis} and
Theorem~\ref{intro4}. Finally, in Section $5$ we prove
Theorem~\ref{intro1}.

\section{Preliminary estimates on $u_p$ and $u_p^-$}
\label{S1}

\subsection{Control of the total energy}
We first recall that solutions of problem~\eqref{pblP} are the
critical points of the energy functional $\E_{p}$ defined on
$H^{1}_{0}(B)$ by
\[\E_{p}(u)=\frac{1}{2}\int_{B}\vert \nabla u\vert^{2}
-\frac{1}{p+1}\int_{B}\vert u\vert^{p+1}.\]

If $u$ is a nodal solution then
\begin{equation}\label{5}
\int_{B}\vert\nabla u^{+}\vert^{2}=\int_{B}\vert u^{+}\vert^{p+1}
\text{ and } \int_{B}\vert\nabla u^{-}\vert^{2} = \int_{B}\vert
u^{-}\vert^{p+1}.
\end{equation}
So, $\E_p (u) = \left(\frac{1}{2} - \frac{1}{p+1}\right)
\int_{B}\abs{\nabla u}^2$.  Moreover, if $u$ is a radial nodal
solution with least energy then $\E_p(u)\leq \E_p (v)$ for any
radial function $v$ belonging to the nodal Nehari set $N_p:=\{u :
u^\pm\neq 0 \text{ and } \int_{B} \abs{\nabla u^\pm}^2 = \int_{B}
\abs{u^\pm}^{p+1}  \}$.

This first proposition gives a control on the energy and, thus, on
the $H^1_0$ and $L^{p+1}$-norms. Let us remark that this result
will be improved in Theorem~\ref{intro3bis}.
\begin{Prop}
\label{prop1}
  $p\int_{B} \abs{\nabla u_p}^2\le C$ for a positive constant $C$ independent of $p$.
\end{Prop}

\begin{proof}
On one hand, let us  consider the unique positive radial solution
$v_p$
 of $-\Delta u = u^p$ in $A_{R_p}:=B\setminus B[0,R_p]$ with $u=0$ on $\partial A_{R_p}$, for
 $R_p\in(0,1)$.\\
 Using some results of~\cite{grossi2} (see also \cite{EK}), by some delicate and nontrivial estimates in~\cite{demarchis} it is proved that
 for $R_p = e^{-\alpha p}$ and $\alpha>0$,
\begin{equation}
\label{est1} p\int_{A} \abs{\nabla v_p}^2 \le 8\pi e
\frac{e^{2\alpha -1}}{\alpha},
\end{equation}
for $p$ sufficiently large.

On the other hand,  let us consider the unique positive radial
solution $\eta_p$ of $-\Delta u = u^p$ in $B(0,R_p)$ and define
$\eta_p(r) = \eta_p(\abs{x})$.\\   We get that $\eta_p(r) =
R_p^{-2/(p-1)} f_p(r/R_p)$ where $f_p$ is the unique positive
solution of $-\Delta u = u^p$ on the unit ball $B$.
In~\cite{grossi},\cite{EK}, it is proved that
 $p\int_{B} \abs{\nabla f_p}^2  = p \int_{B} f_p^{p+1}\to
 8\pi e$ as $p\to +\infty$. So, $p\int_0^1
f_p'(r)^2 r\intd r = p \int_0^1 f_p^{p+1}(r) r \intd r \to 4 e$
when $p\to +\infty$. Thus,
\begin{equation}
\label{est2} p \int_0^{R_p} \eta_p'(r)^2 r\intd r = p
\int_0^{R_p}\eta_p^{p+1}(r) r \intd r = \frac{p}{R_p^{4/(p-1)}}
\int_0^1 f_p'(r)^2 r \intd r = \frac{4
  e + o(1)}{R_p^{4/(p-1)}}.
\end{equation}
For $R_p = e^{-\alpha p}$, we get $4 e^{1+4\alpha}$ and when
$\alpha = \frac{1}{5}$, we get $4 e^{\frac{9}{5}}$.

Then, as $u_p$ is a least energy nodal radial solution and as the
function defined by $v_p$ in $B(0,R_p)$ and $\eta_p$ on $A_{R_p}$
is radial and belongs to the nodal Nehari set,  we get that
\begin{equation*}
p \int_{B} \abs{\nabla u_p}^2  \leq 8\pi  e^{\frac{9}{5}} + 40\pi
e e^{\frac{2}{5}}\approx 339.
\end{equation*}
\end{proof}

\subsection{The nodal line shrinks to $(0,0)$}

In the sequel we will use the well known "radial lemma" due to
Strauss (see~\cite{strauss}).
 Let us denote by
$H_{\text{rad}}(B)$ the subspace of  $H^1_0(B)$ given by  radial
functions.

\begin{Lem}
\label{radiallemma} There exists $c_N>0$ such that
\begin{equation*}
\abs{u(r)} \leq c_N \frac{\norm{u}_{H^1}}{r^{(N-1)/2}}\quad\forall
u\in H_{\text{rad}}(B)\ \hbox{and }r\in(0,1).
\end{equation*}
\end{Lem}

The following lemma shows that $\norm{u_p^\pm}_\infty$ do not go
to $0$.
\begin{Lem}
\label{linf} For any $p>1$ we have that
$\norm{u_p^\pm}_\infty\ge\lambda_1^{\frac{1}{p-1}}$ where
$\lambda_1$ is the first eigenvalue of $-\Delta$ on $B$.
\end{Lem}

\begin{proof}
Using Poincar\'e's inequality, we get
\begin{equation*}
\begin{split}
1 = \frac{\int_{B}\abs{u_p^{\pm}}^{p+1}}{\int_{B}\abs{\nabla
  u_{p}^{\pm}}^2} & \leq
\frac{\norm{u_p^\pm}_\infty^{p-1}\int_{B}(u_p^\pm)^2}{\int_{\Omega}\abs{\nabla
u_{p}^{\pm}}^2} \\
&\leq \norm{u_p^\pm}_\infty^{p-1}
\lambda_{1}^{-1}(\Tilde{\Omega}_{p}^{\pm}),
\end{split}
\end{equation*}
where $\Tilde{\Omega}_{p}^{\pm}$ are the nodal domains of $u_p$.
As $\Tilde{\Omega}_{p}^{\pm}\subset B$, we have
$\lambda_1(\Tilde{\Omega}_{p}^{\pm})\geq \lambda_1$ which ends the
proof.
\end{proof}

\begin{Rem}
\label{rem2} Using Proposition 2.7 of \cite{GGP}, we also get that
$\norm{u_p^\pm}_\infty$ are bounded from above.  So,
Lemma~\ref{linf} implies that there exists $0< a <b$ such that
$a<\norm{u_p^\pm}_\infty <b$ for any $p>1$.
\end{Rem}
Let us denote, as in Section \ref{Section-Intro}, by $r_p$ the
"nodal radius", i.e. $u_p(x)=0$ for $|x|=r_p$. We will prove that
$r_p\rightarrow0$ as $p\rightarrow+\infty$.
\begin{Prop}
$r_p\to 0$ as $p\to +\infty$.
\end{Prop}

\begin{proof}
Assume by contradiction that $r_p \geq r_* > 0$. Then, by the
radial lemma~\ref{radiallemma}, we get
\begin{equation*}
u_p(s_p) \leq C \frac{\norm{u_p}_{H^1}}{s_p^{1/2}},
\end{equation*}
where $u_p(s_p)=||u_p^+||_\infty$.\\
As $\liminf_{p\to \infty} u_p(s_p)\geq 1$ (see Lemma~\ref{linf}),
multiplying by $\sqrt p$, for $p$ large we get
\begin{equation*}
 C \left( p\int_{0}^1 u_p'(r)^2 r \intd
  r\right)^{1/2}\geq\frac{1}{2} \sqrt p r_*^{1/2} \to +\infty
\end{equation*}
which is a contradiction with Proposition~\ref{prop1} .
\end{proof}

Moreover, using the same kind of argument as in the proof of
Proposition~\ref{prop1}, we  obtain the following proposition.
\begin{Prop}
\label{propball} We have, as $p\to +\infty$,
\begin{itemize}
\item  $\max_{x\in B(0,r_p)} \abs{u_p(x)} = r_p^{-2/(p-1)} (\sqrt e +
  o(1))$ ;
\item  $p \int_{0}^{r_p}\abs{u_p(r)}^{p+1}  r \intd r = \frac{4 e + o(1)}{r_p^{4/(p-1)}} $;
\item $pu_p'(r_p) = p r_p^{-1}\int_0^{r_p} \abs{u_p(r)}^pr\intd r =
  \frac{4\sqrt e + o(1)}{r_p^{2/(p-1)+1}}$ ;
\item $p \int_{B(0, r_p)} \abs{u_p(x)}^p \intd x = \frac{C + o(1)}{r_p^{2/p-1}}$.
\end{itemize}
\end{Prop}

\begin{proof}
We  again  consider the function $f_p = r_p^{2/(p-1)} u_p(r_pr)$
which is the ground state of our problem on $B$ and we use
estimates on $f$ given in \cite{grossi,GGP}. In particular, we use
that $p \int_0^1 \abs{\nabla
  f_p(r)}^2r\intd r\to 4e $ and $\norm{f_p}_{\infty}\to
\sqrt e$.
\end{proof}

As $r_p\leq 1$, let us remark that the second  point of
Proposition~\ref{propball} implies that  there exists $0<a$ such
that $a<r_p^{\frac{2}{p-1}}\leq 1$.   From now on, let us define
\begin{equation}
\label{rinfdef} 0< r_\infty := \lim_{p\to \infty}
r_p^{\frac{2}{p-1}}.
\end{equation}
\subsection{Rescaling in the ball $B(0,r_p)$}

\begin{Prop}
\label{rescball} For $\left(\varepsilon_p^-\right)^{-2}:=  p
\norm{u_p^-}_\infty^{p-1}= p \abs{u_p(0)}^{p-1}$,
the rescaled functions $z_p^- : B(0, r_p/\varepsilon_p^-)\to \IR$
defined by
\begin{equation*}
z_p^-(x) = -\frac{p}{\abs{u_p(0)}}(u_p^-(\varepsilon_p^-
\abs{x})+u_p(0))
\end{equation*}
converges to $-U$ in $\mathcal{C}^1_\text{loc}(\IR^2)$, $U$ defined as in \eqref{U}.\\
\end{Prop}

\begin{proof}
We have that  $z_p^-$ satisfies
\begin{equation*}
-\Delta z_p^- = - \left\vert 1- \frac{z_p^-}{p}  \right\vert^{p-1}
\left( 1 - \frac{z_p^-}{p}\right)  \ \ \ \text{ in } B\left(0,
\frac{r_p}{\varepsilon_p^-}\right).
\end{equation*}
So, a classical argument (see~\cite{grossi, EK, GGP}) gives that
$z_p^- \to z_\infty$ in $\C_{\text{loc}}^1 (D)$, where $D$ is the
limit domain of $B(0, r_p/ \varepsilon_p^-)$ as $p\to +\infty$ and
$-\Delta z_\infty =-e^{-z_\infty}$.

Moreover
\begin{equation*}
\begin{split}
\frac{r_p}{\varepsilon_p^-} &= \left( r_p^{2/(p-1)}
  \left((\varepsilon_p^-)^{-2}\right)^{1/(p-1)} \right)^{(p-1)/2}\\
&= \left( (r_\infty + o(1)) p^{1/(p-1)} \abs{u_p(0)}\right)^{(p-1)/2}.\\
\end{split}
\end{equation*}
As $p^{1/(p-1)}\to 1$ and, by Proposition~\ref{propball},
$\abs{u_p(0)}= \frac{\sqrt e + o(1)}{r_\infty}$, we also get
\begin{equation*}
\frac{r_p}{\varepsilon_p^-} = (\sqrt e + o(1))^{(p-1)/2} \to
+\infty.
\end{equation*}
Hence the limit domain $D$ is the whole $\IR^2$. Let us show that
$\int_{\IR^2} e^{-z_\infty}<+\infty$. By Proposition~\ref{prop1},
\begin{equation*}
 \int_{\IR^2} e^{-z_\infty} \leq \liminf \int_{B(0,
  r_p/\varepsilon_p^-)} \left\vert 1 - z_p^-/p \right\vert^{p}  =
\liminf \frac{p}{\abs{u_p(0)}}\int_{B(0,r_p)}\abs{u_p^-}^{p}< +
\infty
\end{equation*}
using H\"older's inequality. Therefore we get the assertion
recalling that the function
$-z_\infty=U:=\log\left(\frac{1}{(1+\frac{1}{8}\abs{x}^2)^2}\right)$
solves the problem ~\eqref{eqq1}.
\end{proof}

\section{ Estimates on $u_p^+$}

In this section we prove the following crucial result on the
convergence of the rescaling of the positive part $u_p^+$.
\begin{Prop}
\label{rescan} For $\left(\varepsilon_p^+\right)^{-2}:= p
\norm{u_p^+}_\infty^{p-1}= p u_p(s_p)^{p-1}$ and $l
:= \lim_{p\to
  +\infty}\frac{s_p}{\varepsilon^{+}_p} >0$, the one variable rescaled function
$z_p^+ :\tilde{A}_p :=
(\frac{r_p-s_p}{\varepsilon_p^+},\frac{1-s_p}{\varepsilon_p^+})
\to \IR$ defined by
\begin{equation}
z_p^+ (r) = \frac{p}{u_p(s_p)}(u_p^+(s_p  + \varepsilon^{+}_p r) -
u_p(s_p))
\end{equation}
converges to a function $\tilde{z}_l(r)$   in
$\mathcal{C}^1_\text{loc}(-l, +\infty)$.   The function
$\tilde{z}_l(r-l)$ solves equation~\eqref{eqq2} for $H:=
\int_{0}^le^{\tilde{z}_l(s-l)} s \intd s$.  Moreover,
$\tilde{z}_l(r-l) =  Z_l(r) = \log \left( \frac{2 \alpha^2
    \beta^\alpha \abs{x}^{\alpha-2}}{(\beta^\alpha +
    \abs{x}^\alpha)^2}\right)  $   for $\alpha = \sqrt{2l^2
  +4}$ and $\beta = \left( \frac{\alpha +
    2}{\alpha-2}\right)^{1/\alpha}l$.
\end{Prop}

We already know that $u_p^+$ is a positive radial solution to
\begin{equation*}
\left\{
\begin{aligned}
-\Delta u &= \abs{u}^{p-1}u, \ \ &\text{ in } A_{r_p}:= \{ r_p < \abs{x} <1\},\\
u&= 0 \ \ &\text{ on } \partial A_{r_p}.
\end{aligned}
\right.
\end{equation*}
As $u_p$ is radial, we have $u_p^+(x)= u_p(r)$ with $r\in
(r_p,1)$. It satisfies  $-u_p'' - \frac{1}{r} u_p' = u_p^p$ in the
interval $(r_p,1)$ with the Dirichlet boundary condition. We get
that $z_{p}^{+}$ satisfies for any $p>1$
\begin{equation}
\label{sym3...} \left\{
\begin{aligned}
-(z_{p}^{+})''(r) - \frac{1}{r + s_p/\varepsilon^{+}_p}
(z_{p}^{+})'(r) &= \left(1 +
  \frac{z_{p}^{+}(r)}{p}\right)^p \text{
  in } \tilde{A}_p,\\
z_{p}^{+}\leq 0, z_{p}^{+}\left(
\frac{r_p-s_p}{\varepsilon^{+}_p}\right)&=
z_{p}^{+}\left(\frac{1-s_p}{\varepsilon^{+}_p}\right)= -p \text{
and } z_{p}^{+}(0) = (z_{p}^{+})'(0)=0.
\end{aligned}
\right.
\end{equation}
So, we have three possibilities:
\begin{equation*}
\frac{r_p-s_p}{\varepsilon^{+}_p}\to -\infty, \ \ \
\frac{r_p-s_p}{\varepsilon^{+}_p}\to 0 \ \ \text{ or }
\frac{r_p-s_p}{\varepsilon^{+}_p}\to -l<0, \quad\hbox{for some
}l>0.
\end{equation*}

The two following results show that  the two first possibilities
cannot happen.
\begin{Lem}
  $\frac{r_p-s_p}{\varepsilon^{+}_p}\to -\infty$ cannot
  happen.
\end{Lem}
\begin{proof}
 Otherwise, as $r_p >0$, we have that $s_p/ \varepsilon^{+}_p
\to + \infty$. Passing to the limit in~\eqref{sym3...}, we get
$z_{p}^{+}\to \tilde{z}_l$ in $\mathcal{C}^1_\text{loc}(\IR)$
where $\tilde{z}_l$ solves  $-\tilde{z}_l'' = e^{\tilde{z}_l}$
with $\tilde{z}_l(0)= \tilde{z}_l'(0)=0$. We know that the unique
solution of this problem is given by
\begin{equation*}
\tilde{z}_l(s) = \log \frac{4 e^{\sqrt 2 s}}{(1+e^{\sqrt 2 s})^2}.
\end{equation*}
Integrating the equation from $r_p$ to $s_p$, we get
\begin{equation*}
-\int_{r_p}^{s_p} (u_p'(r)r)' \intd r =
\int_{r_p}^{s_p}u_p^p(r)r\intd r.
\end{equation*}
As $u_p'(s_p)=0$, by the change of variable $r= s_p +
\varepsilon^{+}_ps$  we get
\begin{equation*}
\begin{split}
 u_p'(r_p)r_p &= \int_{r_p}^{s_p} u^p_p(r)r\intd r \\
&= \varepsilon^{+}_p u_p^p(s_p)
\int_{\frac{r_p-s_p}{\varepsilon^{+}_p}}^0 \left( 1
  + z_{p}^{+}(s)/p\right)^p (s_p + \varepsilon^{+}_p s) \intd s\\
& = \varepsilon^{+}_p s_p u_p^p(s_p)
\int_{\frac{r_p-s_p}{\varepsilon^{+}_p}}^0 \left( 1
  + z_{p}^{+}(s)/p\right)^p \intd s + (\varepsilon_p^+)^2
u_p^p(s_p) \int_{\frac{r_p-s_p}{\varepsilon^{+}_p}}^0 \left( 1
  + z_{p}^{+}(s)/p\right)^p  s \intd s\\
& \geq \varepsilon^{+}_p s_p u_p^p(s_p)
\int_{\frac{r_p-s_p}{\varepsilon^{+}_p}}^0 \left( 1
  + z_{p}^{+}(s)/p\right)^p \intd s\\
&=  \frac{s_p}{\varepsilon^{+}_p}\frac{u_p(s_p)}{p}
\int_{\frac{r_p-s_p}{\varepsilon^{+}_p}}^0 \left( 1
  + z_{p}^{+}(s)/p\right)^p \intd s.\\
\end{split}
\end{equation*}

By Fatou's lemma,  $\int_{\frac{r_p-s_p}{\varepsilon^{+}_p}}^0
\left( 1
  + z_{p}^{+}(s)/p\right)^p \intd s \geq \int_{-\infty}^0 e^{\tilde{z}_l} =
C_0>0$. Since $u_p'(r_p) = \frac{C + o(1)}{p r_p^{(p+1)/(p-1)}}$
by Proposition~\ref{propball}, we get $ \frac{C + o(1)}{
r_p^{2/(p-1)}} \geq \frac{C_0 u_p(s_p) s_p}{2 \varepsilon^{+}_p}$.
This is a contradiction as the right-hand side is not bounded
($\frac{s_p}{\varepsilon^{+}_p}\to +\infty$ and $u_p(s_p)$ stays
away from $0$ by Proposition~\ref{linf}) and the left-hand
side is bounded by Proposition~\ref{propball}.\\
\end{proof}

\begin{Lem}\label{re}
 $\frac{r_p - s_p}{\varepsilon^{+}_p}\to 0$ cannot
  happen.
\end{Lem}
\begin{proof}
 Integrating  $u_p$ between $0$ and $s_p$,   as
$-(u_p'(r)r)' = \abs{u_p(r)}^{p-1}u_p(r)r$, we have
\begin{equation*}
0 = \int_{0}^{s_p} \abs{u_p(r)}^{p-1} u_p(r) r \intd r =
-\int_0^{r_p} \abs{u_p(r)}^p r\intd r + \int_{r_p}^{s_p}
u_p^p(r)r\intd r.
\end{equation*}
So, using Proposition~\ref{propball} and since there exists $0<a$
such that  $a< r_p^{2/(p-1)}\leq 1$, we have $C_2 \geq p
\int_{0}^{r_p} \abs{u_p(r)}^p r \intd r = p \int_{r_p}^{s_p}
u_p^p(r) r\intd r = \frac{4\sqrt e + o(1)}{r_p^{2/(p-1)}} \geq C_1 >0.$\\
Let us consider the two alternatives\\
$i)\frac{s_p}{\varepsilon^{+}_p}\le C$\\
$ii)\frac{s_p}{\varepsilon^{+}_p}\rightarrow+\infty$ as $p\rightarrow+\infty$.\\

In the first case, using the change of variable $r= s_p +
\varepsilon^{+}_p s$, we get
\begin{equation*}
\begin{split}
&p \int_{r_p}^{s_p}u_p^p(r)r\intd r  = p
\frac{s_p}{\varepsilon^{+}_p} (\varepsilon_p^+)^2
\int_{\frac{r_p-s_p}{\varepsilon^{+}_p}}^0 u_p^p (s_p +
\varepsilon^{+}_p s) \intd s + p(\varepsilon_p^+)^2
 \int_{\frac{r_p-s_p}{\varepsilon^{+}_p}}^0 u_p^p(s_p + \varepsilon^{+}_p s)
 s \intd s\\
&= \frac{s_p}{\varepsilon^{+}_p} u_p(s_p)
\int_{\frac{r_p-s_p}{\varepsilon^{+}_p}}^0 \left(
  \frac{u_p(s_p + \varepsilon^{+}_p s) }{u_p(s_p)}\right)^p \intd s+ u_p(s_p)
\int_{\frac{r_p-s_p}{\varepsilon^{+}_p}}^0 \left(
  \frac{u_p(s_p + \varepsilon^{+}_p s) }{u_p(s_p)}\right)^p s\intd s.
\end{split}
\end{equation*}
As $u_p(s_p)$ is bounded (see Remark~\ref{rem2}), we get
\begin{equation*}
p \int_{r_p}^{s_p}u_p^p(r)r\intd r  \leq C \left(
  \frac{s_p-r_p}{\varepsilon^{+}_p}\right)\to 0
\end{equation*}
which gives a contradiction.

In the second case, as $\frac{s_p}{\varepsilon^{+}_p}\to +\infty$,
$z_p\to \tilde{z}_l$ in $\mathcal{C}^1_{\text{loc}}(l, +\infty)$
where $\tilde{z}_l$ solves $-u'' = e^{u}$. In this case, $l =
\lim_{p\to
  +\infty} \frac{r_p-s_p}{\varepsilon^{+}_p}=0$. Then,  putting again $r =
s_p + \varepsilon^{+}_p s $,
\begin{equation*}
\begin{split}
&p \int_{r_p}^1 u_p^p(r) r \intd r \geq p \int_{s_p}^1 u_p^p(r) r
\intd
r\\
&= u_p(s_p) \frac{s_p}{\varepsilon^{+}_p}
\int_{0}^{\frac{1-s_p}{\varepsilon^{+}_p}} \left( 1 +
  \frac{z_p(s)}{p}\right)^p \intd s+ u_p(s_p)
\int_{0}^{\frac{1-s_p}{\varepsilon^{+}_p}} \left( 1 +
\frac{z_p(s)}{p}\right)^p s \intd s.
\end{split}
\end{equation*}
By Fatou's lemma, we get
$\int_{0}^{\frac{1-s_p}{\varepsilon^{+}_p}} \left( 1 +
  \frac{z_p(s)}{p}\right)^p \intd s \geq \int_{0}^{+\infty} e^{\tilde{z}_l(s)}>0$ which
implies that
\begin{equation*}
p \int_{r_p}^1 u_p^p(r) r \intd r\geq  u_p(s_p)
\frac{s_p}{\varepsilon^{+}_p} \int_{0}^{+\infty}
e^{\tilde{z}_l(s)}\intd s + u_p(s_p) \int_{0}^{+\infty}
e^{\tilde{z}_l(s)}s \intd s
\end{equation*}
which is a contradiction as $\frac{s_p}{\varepsilon^{+}_p}\to
+\infty$ and
the left-hand side is bounded.\\
\end{proof}
\begin{rmq}\label{rf}
We observe that the result of Lemma \ref{re} will be crucial in
the sequel to prove that the limit problem for $u^-_p$ is a
singular Liouville equation in $\IR^2\setminus\{0\}$. We stress that in
higher dimension the analogous statement of Lemma \ref{re}
would be false implying therefore that the analogous limit problem
for $u^+_p$ is the same as for $u^-_p$.
\end{rmq}
Now, we consider the "good" case   $\frac{r_p -
s_p}{\varepsilon_p^{+}}\to -l <0$.  The following lemma proves, in
particular, that $s_p\to 0$.
\begin{Lem}
\label{lemlm} Let $l>0$ be such that $\frac{r_p -
s_p}{\varepsilon_p^{+}}\to -l$, then
$\frac{s_p}{\varepsilon_p^+}\to l$ as $p\to +\infty$.
\end{Lem}

\begin{proof}
Let $m = \lim_{p\to +\infty} \frac{s_p}{\varepsilon^{+}_p}$. We
already know that $-l+ o(1) = \frac{r_p -
s_p}{\varepsilon^{+}_p}\geq -\frac{s_p}{\varepsilon^{+}_p} = -m +
o(1)$ which implies that $-l\geq -m$. Arguing by contradiction,
let us assume that $-l > -m$.

On one hand,  as $-(z_{p}^{+})'' -\frac{1}{r +
  s_p/\varepsilon^{+}_p} (z_{p}^{+})' = (1 + z_{p}^{+}/p)^{p}$ in $\left(
  \frac{r_p-s_p}{\varepsilon^{+}_p},\frac{1-s_p}{\varepsilon^{+}_p}\right)$,
integrating we get
\begin{equation*}
z_p'(r)(r+\frac{s_p}{\varepsilon_p})= \int_{r}^0 \left( 1 +
  \frac{z_p}{p}\right)^p (s+\frac{s_p}{\varepsilon_p}) \intd s, \quad\hbox{for }r\in(-l,0).
\end{equation*}
Then,
\begin{equation}
\label{number} \abs{z_p'(r)}\abs{r+m + o(1)}\leq  \int_{-l}^0
(\abs{s}+\abs{m}+ o(1)) \intd s \leq C_1.
\end{equation}
As $-l > -m$ we get  $\abs{r+ m + o(1)} \geq C_2 >0$ which implies
that $\abs{z_p'(r)}\leq C_3$ uniformly in $(-l,0)$ and in $p$ (for
$p$ large).  On the other hand, as $z_p(0)=0$ and
$z_p\left(\frac{r_p-s_p}{\varepsilon_p^+}\right)=-p$, by the mean
value theorem in $(-l,0)$, we get that there exists $t_p\in
(-l,0)$ such that $z_p'(t_p)\geq \frac{p}{l}$ which gives a
contradiction as $z_p'$ is uniformly bounded.
\end{proof}

\textit{Proof of Proposition~\ref{rescan} } : We know that
$z_{p}^{+}\to \tilde{z}_l$ in $\mathcal{C}^1_{\text{loc}}(-l,
+\infty)$ and $\tilde{z}_l$ verifies
\begin{equation}
\label{lemblala} \left\{
\begin{aligned}
- \tilde{z}_l'' - \frac{1}{r+m} \tilde{z}_l' &= e^{\tilde{z}_l} \ \ \ \text{ in } (-l,+\infty),\\
\tilde{z}_l \leq 0, \ \tilde{z}_l(0)&= \tilde{z}_l'(0)=0.
\end{aligned}
\right.
\end{equation}
We already know that $z_{p}^{+}\to \tilde{z}_l$ in compact sets in
$(-l, +\infty)$ and
 $\tilde{z}_l$ satisfies equation~\eqref{lemblala}.   So $Z_l(s):= \tilde{z}_l(s-l)$ solves
\begin{equation}
\label{lemblala2} \left\{
\begin{aligned}
- Z_l'' - \frac{1}{r} Z_l' &= e^{Z_l} \ \ \ \text{ in } (0,+\infty),\\
Z_l \leq 0, &\  Z_l(l)= Z_l'(l)=0.
\end{aligned}
\right.
\end{equation}
Let us compute the solutions of equation~\eqref{lemblala2}.
Setting $v(t) = Z_l(e^t) + 2t$ we get that $-v'' = e^{2t}
e^{Z_l(e^t)} = e^v$ in $(-\infty, +\infty)$. Thus all solutions
are given by
\begin{equation*}
v_{\delta, y}(t) = \log\left(\frac{4}{\delta^2} \frac{e^{\sqrt 2
      (t-y)/\delta}}{(1+ e^{\sqrt 2 (t-y))/\delta})^2} \right)
\end{equation*}
for $\delta >0$, $t\in\IR$ and $y\in\IR$. Hence
\begin{equation}
\label{z1} Z_l(t) =  \log\left(\frac{4}{\delta^2} \frac{e^{\sqrt 2
      (\log t -y)/\delta}}{(1 + e^{\sqrt 2 (\log t -y)/\delta})^2}
\right) - 2 \log t.
\end{equation}

Observing that from $Z_l'(t)=0$ we get  $\frac{1 -\sqrt 2
\delta}{1+ \sqrt
  2\delta} = e^{\frac{\sqrt 2}{\delta} (\log t -y)}$. Moreover,
$Z_l(t)=0$ for $t = \frac{\sqrt{1 -2\delta^2}}{\delta}$. As
$Z_l(l)=Z_l'(l)=0$, we get that $l^2 =
\frac{1-2\delta^2}{\delta^2}$ which implies that $\delta =
\frac{1}{\sqrt {2 + l^2}}$. Inserting those estimates
in~\eqref{z1} gives that $Z_l = \log \left( \frac{2 \alpha^2
    \beta^\alpha \abs{x}^{\alpha-2}}{(\beta^\alpha +
    \abs{x}^\alpha)^2}\right)  $   for $\alpha = \sqrt{2l^2
  +4}$ and $\beta = \left( \frac{\alpha +
    2}{\alpha-2}\right)^{1/\alpha}l$.

To complete the result, we prove that
    \begin{equation*}
\int_{-l}^{+\infty} (\tilde{z}_l'(r)(r+l))'\phi'(r) \intd r =
\int_{-l}^{+\infty} e^{\tilde{z}_l}(r+ l)\phi(r)\intd r +
\int_{0}^ {l} e^{\tilde{z}_l(s-l)} s \intd s \phi(-l)
\end{equation*}
for any $\phi \in \C_{0}^{\infty}([-l,+\infty))$.  Let us  fix a
function $\phi \in\C_0^\infty ([-l, +\infty))$. Multiplying by
$\phi$ the equation solved by $z_p^+$ and integrating
 by parts, we get for $p$ large that
\begin{equation}
\label{16}
\begin{split}
\int_{\frac{r_p-s_p}{\varepsilon^{+}_p}}^{\frac{1-s_p}{\varepsilon^{+}_p}}
(z_{p}^{+})'(r)(r + \frac{s_p}{\varepsilon^{+}_p}) \phi'(r)\intd r
&+ (z_{p}^{+})'\left(
  \frac{r_p-s_p}{\varepsilon^{+}_p} \right) \frac{r_p}{\varepsilon^{+}_p} \phi
\left( \frac{r_p -s_p}{\varepsilon^{+}_p}\right) =\\
&\int_{\frac{r_p-s_p}{\varepsilon^{+}_p}}^{\frac{1-s_p}{\varepsilon^{+}_p}}\left(
1 + \frac{z_{p}^{+}(r)}{p}\right)^p \left( r +
\frac{s_p}{\varepsilon_p}\right)\phi(r) \intd r. \\
\end{split}
\end{equation}
Since\quad
$\int_{\frac{r_p-s_p}{\varepsilon^{+}_p}}^{\frac{1-s_p}{\varepsilon^{+}_p}}
(z_{p}^{+})'(r)(r + \frac{s_p}{\varepsilon^{+}_p}) \phi'(r) \intd
r\to \int_{-l}^{+\infty} \tilde{z}_l'(r) (r+l) \phi'(r)\intd
r$\quad and as\\
$\int_{\frac{r_p-s_p}{\varepsilon^{+}_p}}^{\frac{1-s_p}{\varepsilon^{+}_p}}\left(
1 + \frac{z_{p}^{+}(r)}{p}\right)^p \phi(r)\intd r \to
\int_{-l}^{\infty} e^{\tilde{z}_l(r)}\phi(r)\intd r$, it remains
to compute
$$\lim_{p\to
  +\infty} (z_{p}^{+})'\left(
  \frac{r_p-s_p}{\varepsilon^{+}_p} \right) \frac{r_p}{\varepsilon^{+}_p} \phi
\left( \frac{r_p -s_p}{\varepsilon^{+}_p}\right)$$ to get the
claim.  As $(z_{p}^{+})'(r) = \frac{p\varepsilon^{+}_p}{u_p(s_p)}
u_p' (s_p + \varepsilon^{+}_p r)$, we have
\begin{equation*}
(z_{p}^{+})'(\frac{r_p-s_p}{\varepsilon^{+}_p})
\frac{r_p}{\varepsilon^{+}_p} = \frac{p r_p}{u_p(s_p)} u_p'(r_p).
\end{equation*}
Integrating by parts the equation satisfied by $u_p$,
$u_p'(r_p)r_p = \int_{r_p}^{s_p} u_p^p(r)r\intd r$, we finally get
substituting again $r= s_p + \varepsilon^{+}_p s$
\begin{equation*}
\begin{split}
&\lim_{p\to +\infty} (z_{p}^{+})'\left(
  \frac{r_p-s_p}{\varepsilon^{+}_p}\right)\frac{r_p}{\varepsilon^{+}_p} =
\lim_{p\to + \infty} \frac{p}{u_p(s_p)}
\int_{r_p}^{s_p}u_p^p(r)r\intd
r\\
&= \lim_{p\to +\infty} p (\varepsilon_p^+)^2 u_p(s_p)^{p-1}
\int_{\frac{r_p-s_p}{\varepsilon^{+}_p}}^{0} \left( 1 +
z_{p}^{+}(s)/p\right)^p
\left( \frac{s_p}{\varepsilon^{+}_p} + s\right) \intd s\\
&= \int_{-l}^ {0} e^{\tilde{z}_l (s)} (s+l)\intd s.\\
\end{split}
\end{equation*}
Then for $H = -  \int_{-l}^ {0} e^{\tilde{z}_l (s)} (s+l)\intd s
$, passing to the limit in equation~\eqref{16} we obtain
    \begin{equation*}
\int_{-l}^{+\infty} (\tilde{z}_l'(r)(r+l))'\phi'(r) \intd r =
\int_{-l}^{+\infty} e^{\tilde{z}_l}(r+l)\phi + H \phi(-l)
\end{equation*}
for any $\phi \in \C_{0}^{\infty}([-l,+\infty))$.
\begin{flushright}
$\square$
\end{flushright}

\section{Final estimates and proofs of Theorem~\ref{intro3}, Theorem~\ref{intro3bis} and Theorem~\ref{intro4}}

We start with some preliminary identities.

\begin{Lem}
\label{lemrate..} For any $r\in(0,1),\ u_p'(r) r \log r - u_p(r) =
\int_{r}^{1} s \log (s) u_p^p(s)\intd s$. So, $\int_{s_p}^{1} s
\log (s) u_p^p(s) \intd s = - u_p(s_p)$.
\end{Lem}

\begin{proof}
By the equation, we have $- \int_{r}^1 (u_p'(s)s)' \log s \intd s
= \int_{r}^{1} u_p^p(s) s \log s \intd s$. Integrating by parts
the first integral we derive the first assertion. To get the
second one, we just have to take $r = s_p$.
\end{proof}

\begin{Lem}
\label{lemrate1} We have $$1 = -\frac{1}{2} \log(r_\infty
u_\infty)\int_{0}^l e^{Z_l(t)}t \intd t = -\frac{1}{2}
\log(r_\infty u_\infty) (\alpha-2)$$ where $\alpha$ is given  in
Proposition~\ref{rescan}, $r_\infty$ by the
equation~\eqref{rinfdef} and $u_\infty:= \lim_{p\to \infty}
u_p(s_p)$.
\end{Lem}

\begin{proof}
We first observe that, by Proposition~\ref{rescan} and using the
definiton of $\beta$,
\begin{equation*}
\begin{split}
\int_{0}^l e^{Z_l(t)}t\intd t &= \int_{0}^l 2 \alpha^2
\frac{\beta^\alpha t^{\alpha-1}}{(\beta^\alpha + t^\alpha)^2}\intd t\\
&= 2\alpha \frac{l^\alpha}{\frac{\alpha+2}{\alpha-2}l^\alpha +
  l^\alpha}\\
&= \alpha -2.
 \end{split}
\end{equation*}
Multiplying $-(u_p'(r)r)' = u_p^p(r)r$ by $\log r - \log r_p$ and
integrating by parts leads to
\begin{equation*}
\begin{split}
\int_{r_p}^{s_p} (\log r - \log r_p) u_p^p(r)r\intd r &=
-\int_{r_p}^{s_p} (u_p'(r)r)'
(\log r -\log r_p) \intd r\\
&= \int_{r_p}^{s_p} u_p'(r)r \frac{1}{r}\intd r\\
&= u_p(s_p).
\end{split}
\end{equation*}

Then, as $\frac{s_p}{\varepsilon^{+}_p}\to l$ (see
Lemma~\ref{lemlm}), we observe that
\begin{equation}
\label{lamerate22}
\begin{split}
\log(s_p + \varepsilon^{+}_p s) -
\log r_p &= \log(l \varepsilon^{+}_p + \varepsilon^{+}_p s + o(1)\varepsilon^{+}_p)-\log r_p \\
&= \log(l+s + o(1))+ \log\frac{\varepsilon^{+}_p}{r_p} \\
&= \log(l+ s + o(1)) + \frac{1}{2} \log \frac{1}{p u_p(s_p)^{p-1}r_p^2}\\
&= \log(l+ s + o(1)) + \frac{p-1}{2} \log \frac{1}{p^{1/(p-1)} u_p(s_p) r_p^{2/(p-1)}} \\
&=   \log(l+ s + o(1)) + \frac{p-1}{2} \log \frac{1+o(1)}{r_\infty
u_\infty}.
\end{split}
\end{equation}
With the usual change of variable $r = s_p + \varepsilon^{+}_p s$,
we get, using also \eqref{lamerate22}
\begin{align*}
&u_p(s_p) = \varepsilon^{+}_p
\int_{\frac{r_p-s_p}{\varepsilon^{+}_p}}^{0} (\log(s_p +
\varepsilon^{+}_p s)-\log r_p)\frac{u_p^p(s_p + \varepsilon^{+}_p
  s)}{u_p(s_p)^p} u_p(s_p)^p (s_p + \varepsilon^{+}_p s)\intd s\\
&= (\varepsilon_p^+)^2 u_p(s_p)^p
\int_{\frac{r_p-s_p}{\varepsilon^{+}_p}}^0 \bigl( \log(l+ s +
o(1)) +
  \frac{p-1}{2} \log(\frac{1 + o(1)}{r_\infty u_\infty} )\bigr) (1 + \frac{z_{p}^{+}}{p})^p
(\frac{s_p}{\varepsilon^{+}_p}+ s) \intd s\\
&= \frac{u_p(s_p)}{p} \int_{\frac{r_p-s_p}{\varepsilon^{+}_p}}^0
\bigl( \log(l+ s + o(1)) +
  \frac{p-1}{2} \log(\frac{1+ o(1)}{r_\infty u_\infty})\bigr) (1 + \frac{z_{p}^{+}}{p})^p
(\frac{s_p}{\varepsilon^{+}_p}+ s) \intd s
\end{align*}
which converges to $ 0 +  u_{\infty} \left(\frac{-1}{2}\right)
\log (r_\infty u_\infty)  \int_{-l}^{0} e^{\tilde{z}_l(s)} (l+s)
\intd s$ as $p\to +\infty$. So we obtain the claim.
\end{proof}

\begin{Lem}
\label{lem4} $\frac{4\sqrt e}{r_\infty} = u_\infty (\alpha-2)$.
\end{Lem}
\begin{proof}
As $-\int_{r_p}^{s_p}(u_p'(r)r)' \intd r= \int_{r_p}^{s_p}
u_p^p(r) r\intd r$, $\frac{r_p-s_p}{\varepsilon^{+}_p}\to -l$ and
$\frac{s_p}{\varepsilon^{+}_p}\to l$, putting $r = s_p +
\varepsilon^{+}_p s$, we have
\begin{equation*}
\begin{split}
u_p'(r_p)r_p &= \int_{r_p}^{s_p}u_p^p(r)r\intd r\\
&= \varepsilon^{+}_p \int_{\frac{r_p-s_p}{\varepsilon^{+}_p}}^0
u_p^p(s_p +
\varepsilon^{+}_p s) (s_p + \varepsilon^{+}_p s) \intd s\\
&= \frac{u_p(s_p)}{p}
\int_{{\frac{r_p-s_p}{\varepsilon^{+}_p}}}^{0} \left( 1+
z_{p}^{+}/p\right)^p
(\frac{s_p}{\varepsilon^{+}_p} + s)\intd s\\
&= \frac{u_\infty + o(1)}{p} \left( \int_{-l}^{0}
e^{\tilde{z}_l(s)}(l + s) +
  o(1) \intd s\right)\\
&= \frac{u_\infty+ o(1)}{p} \left( \int_{0}^l e^{Z_l(s)} s\intd s
+
  o(1)\intd s\right).\\
\end{split}
\end{equation*}
Since by Proposition~\ref{propball}, $u_p'(r_p)r_p \approx
\frac{4\sqrt e}{p
  r_p^{2/(p-1)}}$ , we get our statement using Lemma~\ref{lemrate1}.
\end{proof}

\begin{Prop}
\label{proprate1} We have that $r_\infty u_\infty$ is the unique
root of the equation
\begin{equation}
\label{eqroot} 2\sqrt e \log x + x =0.
\end{equation}
\end{Prop}

\begin{proof}
By Lemma~\ref{lemrate1} and Lemma~\ref{lem4}, we get that
$\frac{4\sqrt e}{r_\infty}  = \frac{2u_\infty }{-\log(r_\infty
u_\infty)}$. It implies that $2\sqrt e \log(r_\infty u_\infty) = -
r_\infty u_\infty$ which ends the proof.
\end{proof}

Let us denote by $\Bar{t}$ the unique solution of $2\sqrt e \log t
+ t =0$. Then
\begin{equation}
\label{linkur} r_\infty = \frac{\Bar{t}}{u_\infty}.
\end{equation}

\begin{Rem}
\label{rem3} Lemma~\ref{lem4} and equation~\eqref{linkur} imply
that
\begin{equation*}
\label{alpha} \alpha = 2 + \frac{4\sqrt e}{\Bar{t}}>2.
\end{equation*}
\end{Rem}

Previous results give some links between $\alpha$, $r_\infty$ and
$u_\infty$. So, it is enough to compute exactly $u_\infty$ to be
able to characterize all the other values.   For this, we need
some other preliminary estimates.

\begin{Lem}
\label{ipbounded} $I_p:=\int_{0}^{\frac{1-s_p}{\varepsilon^{+}_p}}
\left(\frac{s_p}{\varepsilon^{+}_p} + t \right)  \left( 1 +
\frac{z_{p}^{+}(t)}{p}\right)^p\intd t$ is bounded.
\end{Lem}

\begin{proof}
We have, substituting $t = \frac{s- s_p}{\varepsilon^{+}_p}$,
\begin{equation*}
\begin{split}
I_p &= \frac{1}{\varepsilon^{+}_p} \int_{s_p}^1 \left(
  \frac{s_p}{\varepsilon^{+}_p} + \frac{s- s_p}{\varepsilon^{+}_p}\right)
\left( 1 + \frac{z_{p}^{+}((s-s_p)/\varepsilon^{+}_p)}{p}\right)^p \intd s\\
&= \frac{1}{(\varepsilon^{+}_p)^2} \int_{s_p}^1 s
\frac{u_p^p(s)}{u_p(s_p)^p} \intd s\\
&= \frac{p}{u_p(s_p)} \int_{s_p}^1 s u_p^p(s)\intd s\\
&\leq C_1 p \left(\int_{s_p}^1 s u_p^{p+1}
\right)^{\frac{p}{p+1}}\leq C_2
\end{split}
\end{equation*}
since  $u_p(s_p) \geq \frac{1}{2}$ by Proposition~\ref{linf} and
$ \int_{s_p}^1 s u_p^{p+1}(s)\intd s$ is bounded by
Proposition~\ref{prop1}.
\end{proof}

\begin{Lem}
\label{lemrate...}
$I_p:=\int_{0}^{\frac{1-s_p}{\varepsilon^{+}_p}}
\left(\frac{s_p}{\varepsilon^{+}_p} + t \right)  \left( 1 +
\frac{z_{p}^{+}(t)}{p}\right)^p\intd t \to \alpha +2$.
\end{Lem}

\begin{proof}

Let us denote by $I_\infty := \lim_{p\to
  +\infty}\int_{0}^{\frac{1-s_p}{\varepsilon^{+}_p}} (s_p/ \varepsilon^{+}_p +
t )  \left( 1 + z_{p}^{+}(t)/p\right)^p\intd t$. We first remark
that $I_\infty \geq \alpha +2 >4$. Indeed, by Fatou's lemma and
Lemma~\ref{lemrate1}, we get
\begin{equation*}
\begin{split}
I_\infty &\geq \liminf_{p\to
  +\infty} \int_{0}^{\frac{1-s_p}{\varepsilon^{+}_p}}  \left((s_p/\varepsilon^{+}_p + s) (1+ \frac{z_{p}^{+}(s)}{p})^p
\right) \intd s\\
& \ge \int_{0}^{+\infty} (l+ s) e^{\tilde{z}_l(s)}\intd s\\
& = \int_{l}^{+\infty} s e^{Z_l(s)}\intd s = \alpha +2.\\
\end{split}
\end{equation*}

By Pohoza\"ev identity we have
\begin{equation*}
\frac{2}{p+1} \int_{B}\abs{u_p(x)}^{p+1} \intd x = \frac{1}{2}
\int_{\partial B}(x\cdot \nu) \left(\frac{\partial
u_p(x)}{\partial
    \nu}\right)^2.
\end{equation*}
So, $\frac{2}{p+1} \int_0^1 \abs{u_p(r)}^{p+1} r\intd r =
\frac{1}{2} u'_p(1)^2$, i.e.\ $p \int_0^1 \abs{u_p(r)}^{p+1}r\intd
r = \frac{p
  (p+1)}{4} u'_p(1)^2$.
Moreover, by the equation,
\begin{equation*}
\begin{split}
-u_p'(1) &= \int_{0}^{s_p} \abs{u_p(r)}^{p-1}u(r) r\intd r +
\int_{s_p}^{1}
\abs{u_p(r)}^{p-1}u(r) r\intd r\\
&=    \int_{s_p}^{1} u_p^{p}(r) r\intd r\\
&= \varepsilon^{+}_p \int_{0}^{\frac{1-s_p}{\varepsilon^{+}_p}}
u_p^p(x_ p +
\varepsilon^{+}_p s) (s_p + \varepsilon^{+}_p s) \intd s\\
&= \frac{u_p(s_p)}{p} \int_{0}^{\frac{1-s_p}{\varepsilon^{+}_p}}
\left( 1+ z_{p}^{+}(s)/p\right)^p
\left( s_p/ \varepsilon^{+}_p + s \right) \intd s\\
&= \frac{u_\infty + o(1)}{p}
\int_{0}^{\frac{1-s_p}{\varepsilon^{+}_p}} (s_p/ \varepsilon^{+}_p
+
t )  \left( 1 + z_{p}^{+}(t)/p\right)^p\intd t.\\
\end{split}
\end{equation*}
Hence,
\begin{equation*}
p \int_0^1 \abs{u_p(r)}^{p+1}r\intd r = \frac{p (p+1)}{4}
\frac{(u_\infty + o(1))^2}{p^2} I_p^2.
\end{equation*}

On the other hand, using Proposition~\ref{propball} and
Lemma~\ref{lemrate1},  we get
\begin{equation}
\label{ipP}
\begin{split}
&p\int_{0}^1 \abs{u_p(r)}^{p+1} r \intd r = p \int_{0}^{r_p}
\abs{u_p(r)}^{p+1} r \intd r + p \int_{r_p}^{s_p}
\abs{u_p(r)}^{p+1} r \intd r + p \int_{s_p}^{1}
\abs{u_p(r)}^{p+1} r \intd r\\
&= \frac{4e + o(1)}{r_p^{4/(p-1)}} + p(\varepsilon_p^+)^2
u_p(s_p)^{p+1} \int_{\frac{r_p-s_p}{\varepsilon^{+}_p}}^ {0}
\left(1 + z_{p}^{+}(s)/p \right)^{p+1}
\left(\frac{s_p}{\varepsilon^{+}_p} + s\right)\intd s\\ &   \ \ \
\ \ \ + p  (\varepsilon_p^+)^2u_p(s_p)^{p+1}
\int_0^{\frac{1-s_p}{\varepsilon^{+}_p}} \left(1 + z_{p}^{+}(s)/p
\right)^{p+1} \left(\frac{s_p}{\varepsilon^{+}_p} + s  \right)\intd s\\
&= \frac{4e + o(1)}{r_\infty^2 + o(1)} + (u_\infty^2 (\alpha -2) +
o(1)) + u_p(s_p)^2 \int_{0}^{\frac{1-s_p}{\varepsilon^{+}_p}}
\left(
  \frac{s_p}{\varepsilon^{+}_p} + s\right) \left( 1+ z_{p}^{+}(s)/p\right)^{p+1}\intd s\\
&=  \frac{4e + o(1)}{\Bar{t}^2+ o(1)}u_\infty^2 + u_\infty^2
(\alpha -2) + o(1) + u_p(s_p)^2
\int_{0}^{\frac{1-s_p}{\varepsilon^{+}_p}} \left(
  \frac{s_p}{\varepsilon^{+}_p} + s\right) \left( 1+ z_{p}^{+}(s)/p\right)^{p+1} \intd
s.\\
\end{split}
\end{equation}

The inequality $\left( 1 + z_{p}^{+}/p \right)\leq 1$ implies
$\left( 1 +
  z_{p}^{+}/p\right)^{p+1}\leq \left(1 + z_{p}^{+}/p\right)^p$. So we get
\begin{equation*}
p\int_{0}^1 \abs{u_p(r)}^{p+1} r \intd r \leq \frac{4e}{\Bar{t}^2
+ o(1)} u_\infty^2 + u_\infty^2 (\alpha -2) + (u_\infty^2 +
o(1))I_p + o(1).
\end{equation*}

Hence Pohoza\"ev identity becomes
\begin{equation*}
\frac{p (p+1)}{4} \frac{(u_\infty + o(1))^2}{p^2} I_p^2 \leq
\frac{4e+ o(1)}{\Bar{t}^2+ o(1)} u_\infty^2 + u_\infty^2 (\alpha
-2) + (u_\infty^2 + o(1))I_p + o(1).
\end{equation*}
Since $I_p$ is bounded by Lemma~\ref{ipbounded}, passing to the
limit as $p\to +\infty$, we obtain
\begin{equation*}
\frac{u_\infty^2}{4} I_\infty^2 \leq
\frac{4e}{\Bar{t}^2}u_\infty^2 + u_\infty^2 (\alpha -2) +
u_\infty^2 I_\infty.
\end{equation*}

Thus, by previous estimates (as $u_\infty >0$),
$\frac{I_\infty^2}{4} -I_\infty \leq \frac{4e}{\Bar{t}^2} +
(\alpha -2) = \frac{(\alpha -2)^2}{4} + (\alpha -2) =
\frac{(\alpha +
  2)^2}{4} - (\alpha + 2)$.   Since the function $\frac{x^2}{4} -4$ is increasing on
$x\geq 2$ and  we already proved that $I_\infty \geq \alpha + 2$,
we directly get that $I_\infty  = \alpha +2$.
\end{proof}

\begin{Lem}
\label{jpbounded} $\frac{J_p}{p}:= \frac{1}{p}
\int_{0}^{\frac{1-s_p}{\varepsilon^{+}_p}} \left(
\frac{s_p}{\varepsilon^{+}_p} + t \right) \left( 1 +
  \frac{z_{p}^{+}(t)}{p}\right)^p \log (l+ t + o(1))\intd t = o(1)$ as $p\to
+ \infty$.
\end{Lem}

\begin{proof}
We prove that $\forall \varepsilon >0, \exists p_0 >0 : \forall p
> p_0, \frac{J_p}{p}< \varepsilon$.

Let us fix $\varepsilon >0$ and choose $R_\varepsilon$ such that
\begin{equation}
\label{reps} \frac{2\alpha \beta^\alpha}{\beta^\alpha +
(R_\varepsilon + l)^\alpha} < \frac{\varepsilon}{3 (\log u_\infty
+ 1)}.
\end{equation}
Then, for $p$ large,
\begin{equation*}
\begin{split}
\frac{J_p}{p} &= \frac{1}{p} \int_{0}^{R_\varepsilon} \left(
\frac{s_p}{\varepsilon^{+}_p} + t \right) \left( 1 +
  \frac{z_{p}^{+}(t)}{p}\right)^p \log (l+ t + o(1))\intd t\\ & \ \ \
\ \ \ \ \ \ \  \ \ \ \   \ \  \  +\frac{1}{p}
\int_{R_\varepsilon}^{\frac{1-s_p}{\varepsilon^{+}_p}} \left(
\frac{s_p}{\varepsilon^{+}_p} + t \right) \left( 1 +
  \frac{z_{p}^{+}(t)}{p}\right)^p \log (l+ t + o(1))\intd t\\
&= \frac{J_p'}{p} + \frac{J_p''}{p}.
\end{split}
\end{equation*}
Since $\left( \frac{s_p}{\varepsilon^{+}_p} + t \right) \left( 1 +
  \frac{z_{p}^{+}(t)}{p}\right)^p \log (l+ t + o(1)) \to (l+ t)
e^{\tilde{z}_l(t)}\log (l+t)$ and $R_\varepsilon$ is fixed, we get
$\frac{J_p'}{p}\to 0$. So, there exists $p_0'$ such that, for any
$p> p_0'$, $\frac{J'_p}{p}< \frac{\varepsilon}{3}$.

Then, for $\frac{J_p''}{p}$, we have
\begin{equation*}
\begin{split}
&\log (l+ t + o(1) ) \leq \log (l +
\frac{1-s_p}{\varepsilon^{+}_p} + o(1)) \leq\\ &
\abs{\log (l
\varepsilon^{+}_p +
  1 - s_p + o(\varepsilon^{+}_p))} + \abs{\log \varepsilon^{+}_p} \leq C_1 +
C_2 p
\end{split}
\end{equation*}
 for large $p$. Indeed, the first term is bounded and for the second
one we have
$$ 2\abs{\log \varepsilon^{+}_p} = \log (p
u_p^{p-1}(s_p))= \log p + (p-1) \log u_p(s_p) \leq C_2p.$$ Hence,
\begin{equation*}
\begin{split}
\frac{J_p''}{p} &\leq \frac{C_1 + C_2 p}{p}
\int_{R_\varepsilon}^{\frac{1-s_p}{\varepsilon^{+}_p}}
\left(\frac{s_p}{\varepsilon^{+}_p} + t \right) \left( 1 +
  \frac{z_{p}^{+}(t)}{p}\right)^p \intd t \\
&= o(1) + C_2
\int_{R_\varepsilon}^{\frac{1-s_p}{\varepsilon^{+}_p}} \left(
\frac{s_p}{\varepsilon^{+}_p} + t \right) \left( 1 +
\frac{z_{p}^{+}(t)}{p}\right)^p \intd t.
\end{split}
\end{equation*}
Finally, using Lemma~\ref{lemrate...},
\begin{equation*}
\begin{split}
\int_{R_\varepsilon}^{\frac{1-s_p}{\varepsilon^{+}_p}} \left(
\frac{s_p}{\varepsilon^{+}_p} + t \right) \left( 1 +
  \frac{z_{p}^{+}(t)}{p}\right)^p \intd t &= \int_{0}^{\frac{1-s_p}{\varepsilon^{+}_p}}
\left( \frac{s_p}{\varepsilon^{+}_p} + t \right) \left( 1 +
  \frac{z_{p}^{+}(t)}{p}\right)^p \intd t\\ & \ \ \ \  - \int_{0}^{R_\varepsilon}
\left( \frac{s_p}{\varepsilon^{+}_p} + t \right) \left( 1 +
  \frac{z_{p}^{+}(t)}{p}\right)^p \intd t
\end{split}
\end{equation*}
which converges to $\alpha + 2 - \int_{0}^{R_\varepsilon} (l+ t)
e^{\tilde{z}_l(t)}\intd t = \alpha + 2 - \int_{l}^{R_\varepsilon +
l} t e^ {Z_l(t)}\intd t$.

As $\int_{l}^{R_\varepsilon + l} t e^ {Z_l(t)}\intd t = \alpha + 2
- \frac{2\alpha \beta^\alpha}{\beta^\alpha + (R_\varepsilon +
  l)^\alpha}$, we get $\frac{J_p''}{p}\leq o(1) + C_2 \left( o(1) +
  \frac{2\alpha \beta^\alpha}{\beta^\alpha + (R_\varepsilon +
    l)^\alpha}\right)$. Hence, by definition of $R_\varepsilon$, there
exists $p_0''$ such that for all $p> p_0''$
\begin{equation*}
\frac{J_p''}{p} \leq \frac{\varepsilon}{3} + C_2
\frac{\varepsilon}{3C_2} = \frac{2\varepsilon}{3}.
\end{equation*}
This ends the proof using $p_0 = \min(p_0', p_0'')$.

\end{proof}

\begin{Prop}
\label{propuinf} $u_\infty = e^{\frac{2}{\alpha+2}}$.
\end{Prop}

\begin{proof}
By Lemma~\ref{lemrate..}, substituting $s= s_p + \varepsilon^{+}_p
t$, we derive
\begin{equation*}
\begin{split}
-u_p(s_p) &= \varepsilon^{+}_p
\int_{0}^{\frac{1-s_p}{\varepsilon^{+}_p}} (s_p +
\varepsilon^{+}_p t ) \log (s_p + \varepsilon^{+}_p t) u_p^p(s_p +
\varepsilon^{+}_p t)\intd t\\
& = (\varepsilon_{p}^+)^2 u_p^p(s_p)
\int_{0}^{\frac{1-s_p}{\varepsilon^{+}_p}} (s_p/ \varepsilon^{+}_p
+ t ) \log (s_p + \varepsilon^{+}_p t) \left( 1 +
z_{p}^{+}(t)/p\right)^p\intd t
.\\
\end{split}
\end{equation*}

Using the same idea as in equation~\eqref{lamerate22}, we get
\begin{equation*}
\log(s_p + \varepsilon^{+}_p t) - \log u_p(s_p) = \log(l+ t +
o(1)) + \frac{1}{2} \log(1/p) - \frac{p+1}{2} \log (u_p(s_p)),
\end{equation*}
i.e.\ $\log(s_p + \varepsilon^{+}_p t) = \log(l+ t + o(1)) +
\frac{1}{2} \log(1/p) - \frac{p-1}{2} \log (u_p(s_p))$. It implies
that
\begin{equation*}
\begin{split}
&-u_p(s_p) = \frac{u_p(s_p)}{p}
\int_{0}^{\frac{1-s_p}{\varepsilon^{+}_p}}\left[(s_p/
\varepsilon^{+}_p + t )  \left( 1 + z_{p}^{+}(t)/p\right)^p \big(
-\frac{p-1}{2} \log
  (u_p(s_p)) +\right.\\
  &\left. \frac{1}{2} \log (1/p)+ \log(l + t + o(1))
\big)\right]\intd t = (-\frac{u_\infty}{2} \log u_\infty +
o(1))I_p + o(1) I_p +
\frac{u_p(s_p)}{p} J_p. \\
\end{split}
\end{equation*}
As $I_p\to \alpha +2$ by Lemma~\ref{lemrate...} and
$\frac{J_p}{p}= o(1)$ by Lemma~\ref{jpbounded}, we get that
$-u_\infty = -\frac{1}{2} u_\infty \log u_\infty (\alpha + 2)$
which ends the proof.
\end{proof}

\textit{Proof of Theorem~\ref{intro3}, Theorem~\ref{intro3bis} and
  Theorem~\ref{intro4}} :
 Proposition~\ref{rescball} and Proposition~\ref{rescan} give
Theorem~\ref{intro3}. Then, combining Proposition~\ref{propuinf}
and Remark~\ref{rem3} we obtain the convergence of
$\norm{u_p^+}_\infty$ (see equation~\eqref{normup+}) which is the
first part of Theorem~\ref{intro3bis}.

Since $r_\infty = \frac{\Bar t}{e^{\frac{2}{\alpha +2}}}$ by
equation~\eqref{linkur}, we get Theorem~\ref{intro4}.

From this result, we directly get the convergence of
$\norm{u_p^-}_\infty$ using Propostion~\ref{propball} (see
equation~\eqref{normup-}) which is the second point of
Theorem~\ref{intro3bis}.

Moreover,  by the equation~\eqref{ipP} and as $I_p\to \alpha +2$,
we get that  $p \int_{B}\abs{\nabla u_p}^2 \to \frac{8\pi
  e}{r_\infty^2} + 4 \pi\alpha  u_\infty^2$.   Using
Proposition~\ref{propuinf}, Remark~\ref{rem3} and
equation~\eqref{linkur}, we deduce~\eqref{ener}.\\
Finally, to prove that $pu_p(x)$ converges to $\gamma\log|x|$, which is the
last part of Theorem~\ref{intro3bis}, let us use the
representation formula
\begin{equation*}
\begin{split}
&pu_p(x) =p\int_B G(x,y)|u_p(y)|^{p-1}u_p(y)dy\\
& =p\left[-\int_{\{|y|\le r_p\}}
G(x,y)|u_p(y)|^pdy+\int_{\{r_p<|y|<1\}} G(x,y)|u_p(y)|^pdy\right]
\end{split}
\end{equation*}
where $G(x,y)$ is the Green function of the unit ball.
For the first term we have
\begin{equation*}
\begin{split}
&\int_{\{|y|<r_p\}} G(x,y)|u_p(y)|^pdy=p\int_0^{2\pi}\int_0^{r_p}G(x,(r\cos\theta,r\sin\theta))|u_p(r)|^prdrd\theta\\
&(\hbox{setting }r=r_p\tau)\\
&=p\int_0^{2\pi}d\theta\int_0^1G(x,(r_p\tau\cos\theta,r_p\tau\sin\theta))
|u_p(r_p\tau)|^p\left(r_p\right)^2\tau d\tau\\
&=\frac
p{r_p^\frac2{p-1}}\int_0^{2\pi}d\theta\int_0^1G(x,(r_p\tau\cos\theta,r_p\tau\sin\theta))
\left|r_p^\frac2{p-1}u_p(r_p\tau)\right|^p\tau d\tau\\
&=\frac p{r_p^\frac2{p-1}}\int_BG(x,r_py)
\left|r_p^\frac2{p-1}u_p(r_p|y|)\right|^p dy\\
&\longrightarrow\frac{8\pi\sqrt e}{r_\infty} G(x,0)=\frac{4\sqrt e}{r_\infty}\log|x|\quad\hbox{in }B\setminus\{0\}.\\
\end{split}
\end{equation*}
For the second term we have, setting $r=\varepsilon_p^+\tau+s_p$, by the third equality in Proposition \ref{propball} and the uniform convergence
of $G(x,y)$ to $G(x,0)$ in $B\setminus\{0\}$,
\begin{equation*}
\begin{split}
&p\int_{\{r_p<|y|<1\}} G(x,y)|u_p(y)|^pdy=p\int_0^{2\pi}\int_{r_p}^1G(x,(r\cos\theta,r\sin\theta))|u_p(r)|^prdrd\theta\\
&=p\int_0^{2\pi}d\theta\int_{\frac{r_p-s_p}{\varepsilon_p^+}}^{\frac{1-s_p}{\varepsilon_p^+}}
G(x,(\varepsilon_p^+\tau+s_p)\cos\theta,(\varepsilon_p^+\tau+s_p)\sin\theta))
|u_p(\varepsilon_p^+\tau+s_p)|^p\left(\varepsilon_p^+\right)^2\left(\tau+\frac{s_p}
{\varepsilon_p^+}\right) d\tau\\
&=||u_p^+||_\infty\int_0^{2\pi}d\theta\int_{\frac{r_p-s_p}{\varepsilon_p^+}}^{\frac{1-s_p}{\varepsilon_p^+}}G(x,(\varepsilon_p^+\tau+s_p)\cos\theta,(\varepsilon_p^+\tau+s_p)\sin\theta))
\left|1+\frac{z_p(\tau)}p\right|^p\left(\tau+\frac{s_p}{\varepsilon_p^+}\right) d\tau\\
\end{split}
\end{equation*}
Using \eqref{normup+}, Lemma \ref{lemrate...}, the fact that
$\frac{r_p-s_p}{\varepsilon_p^+}\rightarrow -l$, $\int_{-l}^\infty
(s+l)e^{Z_l(s)}ds=2\alpha$ and the uniform convergence of
$G(x,(\varepsilon_p^+\tau+s_p)\cos\theta,(\varepsilon_p^+\tau+s_p)\sin\theta))\rightarrow
G(x,0)$ in $B\setminus\{0\}$, as $p\rightarrow+\infty$ we finally
get
\begin{equation*}
\begin{split}
&p\int_{\{r_p<|y|<1\}} G(x,y)|u_p(y)|^pdy\rightarrow4\pi\alpha
e^{\frac{\Bar t} {2(\Bar t + \sqrt e)}}G(x,0)\\
&=2\alpha e^ {\frac{\Bar t}{2(\Bar t + \sqrt e)}} \log|x|
\end{split}
\end{equation*}
So, summing up we get that,
$$pu_p(x)\longrightarrow\left(\frac{4\sqrt e}{r_\infty} +2\alpha e^
{\frac{\Bar t}{2(\Bar t + \sqrt
e)}}\right)\log|x|=\left(4+\frac{12\sqrt e}{\Bar t} \right)e^
{\frac{\Bar t}{2(\Bar t + \sqrt e)}}\log|x|$$ as we wanted to
show.
\begin{flushright}
$\square$
\end{flushright}

\section{Low energy solutions -- asymptotic antisimmetry}

We end by proving \textit{Theorem~\ref{intro1}}. Let us denote by
$w_p$ a family of low energy nodal solutions, i.e. solutions
satisfying \eqref{lowenerg}, having Morse index two.
  In~\cite{GGP}, we have proved that,
under the assumption (B) stated in the introduction, $pw_p$
converges, up to a subsequence, to   $8\pi \sqrt e (G(., x^+) -
G(.,x^-))$ where $G$ is the Green function on $B$ and $x^+$ and
$x^-$ are the limit points of the maximum point $x^+_p$ and the
minimum point $x^-_p$ of $w_p$. Moreover it holds
\begin{equation}
\label{sym...!} \left\{
\begin{aligned}
\frac{\partial G}{\partial x_i}(x^+, x^-)-\frac{\partial
H}{\partial
  x_i}(x^+, x^+)=0,\\
\frac{\partial G}{\partial x_i}(x^-, x^+)-\frac{\partial
H}{\partial
  x_i}(x^-, x^-)=0,
\end{aligned}
\right.
\end{equation}
for $i=1,2$, where $H$ denotes the regular part of $G$. Since the
domain is a ball we have that
$$G(x,y) = -\frac{1}{2\pi} \ln \abs{x-y} +
\frac{1}{2\pi} \ln \abs{y} + \frac{1}{2\pi} \ln \abs{x - y/y^2}$$
 and
$$H(x,y )= G(x,y ) + \frac{1}{2\pi} \ln \abs{x-y}.$$
Since $w_p$ have Morse index $2$, by the symmetry result of
~\cite{pacellaweth} (or \cite{bww} for least energy nodal
solution) we deduce that $w_p$ are foliated Schwarz symmetric,
i.e. they are even with respect to a diameter and monotone in the
polar angle. Moreover, by ~\cite{aftalion}, $w_p$ cannot be radial
and this implies that the maximum point $x^+_p$ and the minimum
point $x^-_p$ are on the same diameter but on different sides with
respect to the centre of the ball. Therefore, up to rotation, we
can assume without loss of generality that $x^+ = (0,a)$ and $x^-
=(0, -b)$ with $a>0$, $b>0$.  Inserting this information in
~\eqref{sym...!} we get the system
\begin{equation*}
\label{sym} \left\{
\begin{aligned}
\frac{-1}{a+b} + \frac{b}{ab+1} - \frac{a}{a^2-1}=0,\\
\frac{1}{a+b}  - \frac{a}{ab+1} -\frac{b}{1 -b^2}=0.
\end{aligned}
\right.
\end{equation*}
whose unique solution is $a=b = \sqrt{-2 + \sqrt 5}$. Hence the
points
 $x^+$ and $x^-$ are antipodal and so the limit function $G(.,x^+)-G(.,x^-)$ is
 even with respect to the diameter passing through $x^+$ and $x^-$ and odd with
 respect to the  orthogonal diameter. Then the assertion of Theorem~\ref{intro1} is proved.
\begin{flushright}
$\square$
\end{flushright}

\begin{Rem}
Let us consider the least energy nodal solutions $\tilde w_p$ of
\eqref{pblP}. By \cite{GGP} we know that they satisfy
\eqref{lowenerg} and by
 \cite{bartweth} we know that they have Morse index two.
 Therefore if we knew that they satisfy condition $(B)$, Theorem~\ref{intro1}
 would apply and we could claim that $\tilde w_p$ are asymptotically antisymmetric with
  respect to a diameter. We believe that this is true but so far we have not been able
  to prove $(B)$ for this kind of solutions.
\end{Rem}

\bibliography{paperGGP}

\begin{thebibliography}{10}

\bibitem{grossi}
Adimurthi and Massimo Grossi.
\newblock Asymptotic estimates for a two-dimensional problem with polynomial
  nonlinearity.
\newblock {\em Proc. Amer. Math. Soc.}, 132(4):1013--1019 (electronic), 2004.

\bibitem{aftalion}
Amandine Aftalion and Filomena Pacella.
\newblock Qualitative properties of nodal solutions of semilinear elliptic
  equations in radially symmetric domains.
\newblock {\em C. R. Math. Acad. Sci. Paris}, 339(5):339--344, 2004.

\bibitem{bartweth}
Thomas Bartsch and Tobias Weth.
\newblock A note on additional properties of sign changing solutions to
  superlinear elliptic equations.
\newblock {\em Topol. Methods Nonlinear Anal.}, 22(1):1--14, 2003.

\bibitem{bww}
Thomas Bartsch, Tobias Weth, and Michel Willem.
\newblock Partial symmetry of least energy nodal solutions to some variational
  problems.
\newblock {\em J. Anal. Math.}, 96:1--18, 2005.

\bibitem{benayed1}
Mohamed Ben~Ayed, Khalil El~Mehdi, and Filomena Pacella.
\newblock Classification of low energy sign-changing solutions of an almost
  critical problem.
\newblock {\em J. Funct. Anal.}, 250(2):347--373, 2007.

\bibitem{demarchis}
Francesca De~Marchis, Isabella Ianni, and Filomena Pacella.
\newblock Sign changing solutions to lane emden problems with interior nodal
  line and semilinear heat equations.
\newblock {\em J. Differential Equations}, 2013.

\bibitem{EK}
Khalil El~Mehdi and Massimo Grossi.
\newblock Asymptotic estimates and qualitative properties of an elliptic
  problem in dimension two.
\newblock {\em Adv. Nonlinear Stud.}, 4(1):15--36, 2004.

\bibitem{grossi2}
Massimo Grossi.
\newblock Asymptotic behaviour of the {K}azdan-{W}arner solution in the
  annulus.
\newblock {\em J. Differential Equations}, 223(1):96--111, 2006.

\bibitem{GGP}
Massimo Grossi, Christopher Grumiau, and Filomena Pacella.
\newblock Lane-emden problems: Asymptotic behavior of low energy nodal
  solutions.
\newblock {\em Ann. Inst. H. Poincar\'e Anal. Non Lin\'eaire}, 30(1):121--140,
  2013.

\bibitem{pacellaweth}
Filomena Pacella and Tobias Weth.
\newblock Symmetry of solutions to semilinear elliptic equations via {M}orse
  index.
\newblock {\em Proc. Amer. Math. Soc.}, 135(6):1753--1762 (electronic), 2007.

\bibitem{pistoiaweth}
Angela Pistoia and Tobias Weth.
\newblock Sign changing bubble tower solutions in a slightly subcritical
  semilinear {D}irichlet problem.
\newblock {\em Ann. Inst. H. Poincar\'e Anal. Non Lin\'eaire}, 24(2):325--340,
  2007.

\bibitem{strauss}
Walter~A. Strauss.
\newblock Existence of solitary waves in higher dimensions.
\newblock {\em Comm. Math. Phys.}, 55(2):149--162, 1977.

\end{thebibliography}
\bibliographystyle{plain}

\end{document}